\documentclass[10pt]{article}

\usepackage[T1]{fontenc}

\usepackage{amsmath}
\usepackage{amsthm}
\usepackage{amssymb}
\usepackage{latexsym}

\usepackage{enumitem}

\usepackage[small, pagestyles]{titlesec}
\usepackage[small, it]{caption}

\usepackage[square,comma,numbers,sort&compress]{natbib}

\theoremstyle{plain}
\newtheorem{theorem}{Theorem}[section]
\newtheorem{proposition}[theorem]{Proposition}

\newtheorem{lemma}[theorem]{Lemma}

\renewenvironment{abstract}
  {
    \begin{list}{}%
      {\setlength{\rightmargin}{1in}%
       \setlength{\leftmargin}{1in}}%
    \item[]\ignorespaces\begin{small}
  }
  {
    \end{small}\unskip\end{list}
  }

\newpagestyle{main}[\small]{
  \headrule
  \sethead[\usepage][][]
          {\sc Chains and unique transitive orientations of prime graphs}{}{\usepage}
}

\titleformat{\section}
  {\large\sc}
  {\thesection.}{1em}{}

\title{\sc Chains and unique transitive orientations of prime graphs}

\author{
  Robert Brignall%
  \footnote{The Open University, Milton Keynes, England, UK.},
  \quad
  Sean Mandrick%
  \footnote{University of Florida, Gainesville, Florida, USA.},
  \quad
  and
  \quad
  Vincent Vatter%
  \footnotemark[2]
}

\date{\today}

\begin{document}

\maketitle

\setcounter{footnote}{2}

\begin{abstract}
We give a short, conceptual proof that prime graphs have at most two transitive orientations, a much-quoted result of Gallai. Our proof uses chains, introduced by Chudnovsky, Kim, Oum, and Seymour, which provide a transparent characterization of primality. Transitivity induces a forcing relation on edges; using chains, we show that any two edges of a prime graph are equivalent under this relation, and thus any transitive orientation is unique up to reversal.
\end{abstract}

\pagestyle{main}

\section{Introduction}
\label{sec:intro}

In his seminal 1967 paper \emph{Transitiv orientierbare Graphen}~\cite{gallai:transitiv-orien:}, Gallai established the theory of modular decomposition in graphs and proved, among many results, that prime graphs have at most two transitive orientations (Theorem~\ref{thm:gallai} below). Gallai's argument is part of a much larger structural analysis. Our aim in this note is to give a short self-contained proof using chains, introduced by Chudnovsky, Kim, Oum, and Seymour~\cite{chudnovsky:unavoidable-ind:}. Chains provide a local characterization of primality (Proposition~\ref{prop:chains-exist}), and their use exposes the mechanism underlying Gallai's result: transitivity induces a relation on edges, and primality ensures that all edges are equivalent under this relation.

We now set up the definitions required to state the main result. Throughout, all graphs are finite, simple, and nonempty. A \emph{graph} $G$ has vertex set $V(G)$ and edge set $E(G)$, and we write $u\sim v$ to mean that $uv$ is an edge of $G$ (and $u\nsim v$ if it is not). Given a set $S\subseteq V(G)$ and a vertex $v\notin S$, we say that $v$ \emph{agrees} on $S$ if $v$ is adjacent to all of the vertices of $S$ or to none of them, and that $v$ \emph{disagrees} on $S$ otherwise. A \emph{module} of $G$ is a set $M\subseteq V(G)$ such that every vertex outside~$M$ agrees on $M$. Every set of size $0$, $1$, or $|V(G)|$ is a \emph{trivial module}; a graph is \emph{prime} if it has no nontrivial modules.

A \emph{transitive orientation} of $G$ is an orientation of its edges such that whenever $a\to b$ and $b\to c$ are arcs, so is $a\to c$. Given three vertices $a$, $b$, $c$ forming an induced path $a\sim b\sim c$ with $a\nsim c$, transitivity permits only the orientations $a\leftarrow b\rightarrow c$ and $a\rightarrow b\leftarrow c$, so the orientation of either edge forces that of the other.

This observation motivates the following relations on $E(G)$. For edges $e$ and $f$, write $e\wedge f$ if $e=f$, or if $e$ and $f$ share an endpoint and the other two endpoints are nonadjacent. We define an equivalence relation on $E(G)$ by declaring that two edges $e$ and $f$ are \emph{equivalent} if they are connected by a sequence $e=e_0\wedge e_1\wedge\cdots\wedge e_t=f$. The equivalence classes of this relation are the \emph{edge classes} of $G$. Since the $P_3$ forcing above propagates along such sequences, in any transitive orientation, the direction of one edge determines the directions of all edges in its edge class.

\begin{theorem}[Gallai~{\cite[Satz~1.8, Item~5]{gallai:transitiv-orien:}}\footnote{Theorem~3.1.8, Item~5 in the translation~\cite{gallai:a-translation-o:}.}]
\label{thm:gallai}
If a graph is prime, then any two of its edges belong to the same edge class. In particular, a prime graph has at most two transitive orientations, and if there are two, they are reversals of each other.
\end{theorem}

Note that $K_1$ and $\overline{K}_2$ are the only edgeless prime graphs, and for these the trivial orientation is the unique transitive orientation. Every other prime graph has either no transitive orientations or precisely two, one the reversal of the other.

In Section~\ref{sec:chains} we record some elementary properties, and recall the characterization of primality in terms of chains. In Section~\ref{sec:the_proof} we prove Theorem~\ref{thm:gallai}.

\section{Chains}\label{sec:chains}

A \emph{chain} in a graph $G$ is a sequence $p_1, p_2, \dots, p_m$ of~${m\ge 2}$ distinct vertices such that for all~${2\le i \le m}$, the vertex $p_i$ is either
\begin{enumerate}[nosep]
	\item[(a)] \emph{pendant}, meaning that $p_i \sim p_{i-1}$ and $p_i \nsim p_1, p_2, \dots, p_{i-2}$; or
	\item[(b)] \emph{co-pendant}, meaning that $p_i \nsim p_{i-1}$ and $p_i \sim p_1, p_2, \dots, p_{i-2}$.
\end{enumerate}
In particular, $p_i$ disagrees on $\{p_1, \dots, p_{i-1}\}$ for all $i\ge 3$. Chains%
\footnote{These are unrelated to the ``chains'' of the English translation of Gallai~\cite{gallai:a-translation-o:}, or the ``Ketten'' of Gallai's original German version~\cite{gallai:transitiv-orien:}.}
were defined by Chudnovsky, Kim, Oum, and Seymour~\cite{chudnovsky:unavoidable-ind:}.
They are the graphical analogues of \emph{proper pin sequences}, introduced independently and earlier by Brignall, Huczynska, and Vatter~\cite{brignall:decomposing-sim:} in the context of permutations. The following two results are immediate from the definition.

\begin{proposition}
\label{prop:chains-swap-first-two-elements}
If $p_1, p_2, \dots, p_m$ is a chain in $G$, then so is $p_2, p_1, p_3, \dots, p_m$.
\end{proposition}

\begin{proposition}
\label{prop:chains-in-complement}
If $p_1, p_2, \dots, p_m$ is a chain in $G$, then it is also a chain in the complement~$\overline{G}$.
\end{proposition}

Our next result requires a short argument.

\begin{proposition}
\label{prop:disconnected_chains}
For any chain $p_1, p_2, \dots, p_m$, the vertices $p_1, \dots, p_m$ induce either a connected graph or a $P_1 \cup P_{m-1}$, with $p_1$ or $p_2$ being isolated.
\end{proposition}

\begin{proof}
Assume that $m\ge 3$, as otherwise there is nothing to prove. If there is an index $k\ge 2$ such that $p_1,p_2,\dots,p_k$ induce a connected graph, then adding the later vertices one at a time preserves connectedness, since each $p_i$ for $i\ge 3$ is adjacent to at least one of $p_1,p_2,\dots,p_{i-1}$. Suppose instead that $p_1,p_2,\dots,p_k$ induce a disconnected graph for every $2\le k\le m$. In particular, $p_1\nsim p_2$, and $p_i$ must be pendant for all $i\ge 4$, since otherwise $p_i$ is adjacent to all of $p_1,\dots,p_{i-2}$, while $p_{i-1}$ has a neighbor among these vertices, making the prefix connected. Regardless of whether $p_3$ is pendant or co-pendant, it follows that $p_1,p_2,\dots,p_k$ induce a $P_1\cup P_{k-1}$ for all $2\le k\le m$, with $p_1$ or $p_2$ isolated.
\end{proof}

Chains can be seen as generalizations of paths: just as a graph is connected if and only if every pair of vertices is linked by a path, the following result shows that chains play an analogous role for primality. As the proof is short, we include it to keep this note self-contained.

\begin{proposition}[Chudnovsky, Kim, Oum, and Seymour~{\cite[Proposition~2.1]{chudnovsky:unavoidable-ind:}}]
\label{prop:chains-exist}
A graph is prime if and only if for every three distinct vertices $u$, $v$, and $w$, there is a chain starting with $u, v$ and ending with $w$.
\end{proposition}

\begin{proof}
Let $G$ be a graph. First suppose that $G$ is not prime, and let~$M$ be a nontrivial module with distinct vertices $u, v \in M$. Every chain with its first two vertices inside~$M$ must stay inside~$M$: if $p_1, p_2, \dots, p_k, p_{k+1}$ is a chain with its first $k\ge 2$ vertices in $M$, then $p_{k+1}$ must disagree on $\{p_1, \dots, p_k\} \subseteq M$, so $p_{k+1}$ cannot lie outside $M$. Since $M$ is a nontrivial module, it follows that there is a vertex outside it that cannot be reached by a chain starting with $u, v$.

Conversely, suppose that $G$ is prime and let $u$ and $v$ be distinct vertices. Define~$R$ to be the set of all vertices that appear in some chain starting with $u, v$; note that $u$ and $v$ themselves belong to~$R$. We show that $R$ is a module. Let $z\notin R$ and $w\in R\setminus\{u\}$. Since $w$ appears in a chain starting with $u,v$, truncation shows that it ends such a chain, so take a chain $p_1=u, p_2=v, \dots, p_m=w$. We argue by induction that $z$ agrees on $\{p_1,\dots,p_k\}$ for each $k\le m$. For $k=1$ this is trivial, and for~${k\ge 2}$, if~$z$ agrees on $\{p_1,\dots,p_{k-1}\}$, then $z$ must have the same adjacency to $p_k$ as to $p_1,\dots,p_{k-1}$ (otherwise $p_1,\dots,p_k,z$ would be a chain, but then $z$ would lie in $R$). Taking $k=m$, this shows that $z$ has the same adjacency to $w$ as to $u$. Since $w\in R\setminus\{u\}$ was arbitrary, $z$ agrees on $R$, proving that $R$ is a module. Since $G$ is prime and $|R|\ge 2$, this forces $R=V(G)$, and the result follows.
\end{proof}

\section{Proof of Gallai's Theorem~\ref{thm:gallai}}
\label{sec:the_proof}

To prove Theorem~\ref{thm:gallai}, it suffices to show that any two distinct incident edges belong to the same edge class. Such a pair joins three vertices that induce either a $P_3$ or a $K_3$. The $P_3$ case is immediate from the definition of edge classes, while for $K_3$ we construct a chain using Proposition~\ref{prop:chains-exist}, and then use the chain's structure to propagate the $\wedge$-relation. This propagation is carried out in the following two lemmas, the first of which is auxiliary to the second.

\begin{lemma}
\label{lemma:p1p2_pspm_same_edge_class}
Let $p_1, p_2, \dots, p_m$ be a chain in a graph with $p_1 \sim p_2$, and let $s$ be the greatest index~${i<m}$ with $p_i\sim p_{i+1}$. Then, the edge $p_s p_m$ exists and lies in the same edge class as the edge~$p_1 p_2$.
\end{lemma}

\begin{proof}
Our strategy is to extend the chain one vertex at a time, showing that each extension preserves membership in the edge class of $p_1 p_2$.

We proceed by induction on $m$. If $m = 2$, then the result is trivial. Now suppose that the result holds for a chain $p_1,p_2,\dots,p_m$. Thus, defining $s={\max\{1\le i< m : p_i\sim p_{i+1}\}}$, we know that the edges $p_s p_m$ and $p_1 p_2$ exist and lie in the same edge class. We consider the two different ways an additional vertex $p_{m+1}$ can extend this chain.

First suppose that $p_{m+1}$ is pendant, so $p_{m+1}\sim p_m$ and $p_{m+1}\nsim p_1,p_2,\dots,p_{m-1}$. Then, $m$ is the maximum index $1\le i< m+1$ such that $p_i\sim p_{i+1}$, and $p_mp_{m+1}$ is the edge we seek to place in the same edge class as $p_1 p_2$. By induction, $p_s p_m$ lies in the same edge class as $p_1 p_2$, and $p_s p_m\wedge p_m p_{m+1}$ because $p_{m+1}\nsim p_s$. Therefore, $p_m p_{m+1}$ lies in the same edge class as $p_1 p_2$, as desired.

Next suppose that $p_{m+1}$ is co-pendant, so $p_{m+1}\nsim p_m$ and $p_{m+1}\sim p_1,p_2,\dots,p_{m-1}$. Then, $s$ is still the maximum index $1\le i< m+1$ such that $p_i\sim p_{i+1}$, and $p_sp_{m+1}$ is the edge we seek to place in the same edge class as $p_1 p_2$. By induction, $p_s p_m$ lies in the same edge class as $p_1 p_2$, and $p_s p_m\wedge p_s p_{m+1}$ because $p_m\nsim p_{m+1}$. Therefore, $p_s p_{m+1}$ lies in the same edge class as $p_1 p_2$, completing the proof.
\end{proof}

\begin{lemma}
\label{lemma:final}
Let $p_1, p_2, \dots, p_m$ be a chain such that $p_1, p_2, p_m$ induce a $K_3$. Then, at least one of the edges $p_1 p_m$ or $p_2 p_m$ lies in the same edge class as $p_1 p_2$.
\end{lemma}

\begin{proof}
Note that $p_m$ is co-pendant because $p_1, p_2, p_m$ induce a $K_3$. Thus $p_m\nsim p_{m-1}$ and $p_m\sim p_1,p_2,\dots,p_{m-2}$, and we must have $m\ge 4$ because $p_m$ is adjacent to both $p_1$ and $p_2$.

By Lemma~\ref{lemma:p1p2_pspm_same_edge_class}, there is an index $1\le j\le m-2$ such that $p_j p_m$ lies in the same edge class as $p_1 p_2$. Choose the least such index $j$. If $j\ge 3$, then by the definition of a chain, $p_j$ has a non-neighbor~$p_i$ with $i<j$. Since $p_m$ is adjacent to both $p_i$ and $p_j$ but they are not adjacent to each other, we have $p_j p_m \wedge p_i p_m$. It follows that $p_i p_m$ lies in the same edge class as $p_1 p_2$, but this contradicts the minimality of $j$. Therefore we must have $1\le j\le 2$, and the result follows.
\end{proof}

We now have all we need to prove Theorem~\ref{thm:gallai}.

\newtheorem*{theoremgallai}{Theorem~\ref{thm:gallai}}
\begin{theoremgallai}[Gallai~{\cite[Satz~1.8, Item~5]{gallai:transitiv-orien:}}]
If a graph is prime, then any two of its edges belong to the same edge class. In particular, a prime graph has at most two transitive orientations, and if there are two, they are reversals of each other.
\end{theoremgallai}

\begin{proof}
Let $G$ be a prime graph. If $G$ has fewer than three vertices, there is nothing to prove. We may therefore assume that $G$ has at least three vertices, and so $G$ is connected (every disconnected graph on at least three vertices has a nontrivial module).

Because $G$ is connected, any two edges of~$G$ are joined by a sequence of incident edges, so it suffices to prove that every pair of distinct incident edges lies in the same edge class. Let $uv$ and~$vw$ be distinct incident edges of~$G$. If $u, v, w$ induce a $P_3$, then $uv\wedge vw$ and we are done. Thus we may assume that $u, v, w$ induce a $K_3$. By Proposition~\ref{prop:chains-exist}, there is a chain starting with $u, v$ and ending with $w$, and Lemma~\ref{lemma:final} implies that at least one of the edges $uw$ or $vw$ lies in the same edge class as~$uv$. We are done if this edge is $vw$, so we may assume that $uw$ lies in the same edge class as~$uv$.

Applying Proposition~\ref{prop:chains-exist} again, there is a chain starting with $v, w$ and ending with~$u$. By Lemma~\ref{lemma:final}, at least one of the edges $uv$ or $uw$ lies in the same edge class as~$vw$. Since $uw$ and $uv$ already lie in the same edge class, we conclude that $vw$ lies in the same edge class as $uv$, completing the proof.
\end{proof}

\setlength{\bibsep}{4pt}

%
%
%
%
%


%
%
%
%
%

\end{document}